\documentclass[12pt,reqno]{amsart}
\usepackage{amsthm,amsmath,amssymb}
\usepackage{graphicx}
\usepackage{float}
\usepackage[colorlinks=true,
linkcolor=blue,
urlcolor=blue,
citecolor=blue]{hyperref}
\usepackage[T1]{fontenc}
\usepackage{mathtools}
\usepackage{relsize}
\usepackage{ytableau}
\usepackage{tikz}
\usepackage{enumerate}
\usepackage[toc,page]{appendix}
\usepackage{lscape}
\usepackage{multirow}
\usepackage{adjustbox,expl3,etoolbox} 
\usepackage[margin = 2.8cm]{geometry}
\usepackage{diagbox}
\usepackage{xcolor}
\usepackage{longtable}
\usepackage{comment}
\usepackage{amsaddr}
\usepackage{todonotes}

\newtheorem{thm}{Theorem}[section]
\newtheorem{lem}[thm]{Lemma}
\newtheorem{prop}[thm]{Proposition}
\theoremstyle{definition}

\newtheorem{cor}[thm]{Corollary}

\newtheorem{claim}{Claim}

\newtheorem{hyp}{Hypothesis}

\DeclareMathOperator\Aut{Aut}

\DeclareMathOperator\sym{Sym}
\DeclareMathOperator\alt{Alt}

\newcommand\nv{r} 

\newcommand{\pgl}[2]{\operatorname{PGL}_{#1}(#2)}
\newcommand{\psl}[2]{\operatorname{PSL}_#1(#2)}

\newcommand{\gl}[2]{\operatorname{GL}_{#1}(#2)}

\newcommand{\sln}[2]{\operatorname{SL}_{#1}(#2)}

\newcommand{\dih}[1]{\operatorname{D}_{#1}}

\newcommand{\itbf}[1]{{\bf{{\emph{{#1}}}}}}

\letcs\replicate{prg_replicate:nn}

\begin{document}	
	
	\title[]{The intersection density of cubic arc-transitive graphs with $2$-arc-regular full automorphism group equal to $\pgl{2}{q}$}
	
	\author[K. Meagher]{Karen Meagher}
	\address{Department of Mathematics and Statitstics, University of Regina,\\ 3737 Wascana Parkway, Regina, Saskatchewan, Canada}\email{\href{mailto:karen.meagher@uregina.ca}{karen.meagher@uregina.ca}}
	\author[A.S. Razafimahatratra]{Andriaherimanana Sarobidy Razafimahatratra}
	\address{Fields Institute for
		Research in Mathematical Sciences,\\ 222 College St, Toronto, Canada}
	\email{\href{mailto:sarobidy@phystech.edu}{sarobidy@phystech.edu}}

	
	
	\date{\today}
	
	\maketitle
	
	\begin{abstract}
		The \emph{intersection density} of a transitive permutation group $G\leq \sym(\Omega)$ is the ratio between the largest size of a subset of $G$ in which any two agree on at least one element of $\Omega$, and the order of a point-stabilizer of $G$. In this paper, we determine the intersection densities of the automorphism group of the arc-transitive graphs admitting a $2$-arc-regular full automorphism group $G^* = \pgl{2}{q}$ and an arc-regular subgroup of automorphism $G = \psl{2}{q}$.
	\end{abstract}
	
	
	\section{Introduction}	
	Given a transitive permutation group $G\leq \sym(V)$ where $V$ is a finite non-empty set, a subset $\mathcal{F}\subset G$ is \itbf{intersecting} if for any $f, g \in \mathcal{F}$, there exist $v \in V$ such that $f(v) = g(v)$. If $\mathcal{F} \subset G$ is intersecting and $g\in \mathcal{F}$, then the set $g^{-1}\mathcal{F} = \{ g^{-1}x: x\in \mathcal{F} \}$ is an intersecting set containing the identity of $G$. Therefore, we may assume without of generality that any intersecting set contains the identity of $G$. The \itbf{intersection density} of $G\leq \sym(V)$ is the rational number
	\begin{align*}
		\rho(G) = \max\frac{\left\{ |\mathcal{F}|: \mathcal{F} \subset G \mbox{ is intersecting} \right\}}{|G|/|V|}.
	\end{align*}
	If $G_v$ is the stabilizer of $v\in V$ in $G$, then $G_v$ is intersecting of size $\frac{|G|}{|V|}$, and so $\rho(G)\geq 1$.
	
	The notion of intersection density was first introduced in \cite{li2020erd} to measure how large the intersecting sets in a given transitive group can be compared to its point stabilizers. The majority of the work prior to \cite{li2020erd} focused on transitive groups with intersection density equal to $1$, see \cite{Frankl1977maximum,godsil2009new,meagher2021erdHos,meagher2011erdHos,meagher2016erdHos} for example. In the past few years, the study of intersection densities of transitive groups has become a very active research area. For instance, it was shown in \cite{meagher180triangles} that the intersection density of a transitive group is at most a third of the order of the group, and it was recently shown in \cite{cazzola2025kronecker} that there are only four transitive groups with intersection densities attaining this upper bound. In \cite{behajaina2024intersection,hujdurovic2022intersection,hujdurovic2021intersection,li2020erd}, the intersection densities of transitive groups with prescribed degrees were studied. The notion of intersection density was also recently extended to vertex-transitive graphs in \cite{kutnar2023intersection,meagher2024intersection}.
	
	In this paper, we study the intersecting densities of certain automorphism groups of cubic arc-transitive graphs. Recall that a graph $X = (V,E)$ is cubic if every vertex has degree $3$. Moreover, it is called \itbf{arc-transitive} if for any $(u,v)$ and $(u^\prime,v^\prime)$ such that $\{u,v\},\{u^\prime,v^\prime\} \in E$, there exists $g\in \Aut(X)$ such that $(u^\prime,v^\prime) = g(u,v) = (g(u),g(v))$. In general, an \itbf{$s$-arc} of $X$ is an $(s+1)$-tuple $(u_0,u_1,u_2,\ldots,u_{s})$ with distinct vertices such that $\{u_i,u_{i+1}\} \in E$ for $0\leq i\leq s-1$. The graph $X$ is called \itbf{$s$-arc-transitive} if its automorphism group acts transitively on $s$-arcs, and \itbf{$s$-arc-regular} it acts regularly on $s$-arcs. 
	
	Given a vertex-transitive graph $X = (V,E)$, the \itbf{intersection density array} of $X$ is the increasing array 
	\begin{align}
		\rho(X) := [\rho_1,\rho_2,\ldots,\rho_t]
	\end{align}
	such that for any transitive subgroup $H\leq \Aut(X)$, there exists $j\in \{1,2,\ldots,t\}$ such that $\rho(H) = \rho_j$, and for any $j\in \{1,2,\ldots,t\}$, there is at least one transitive $H \leq \Aut(X)$ whose intersection density is $\rho_j$. For example, the intersection density array of the Petersen graph is $[1,2]$ since its full automorphism group is $\sym(5)$, and the only other transitive subgroup is $\alt(5)$. The intersection densities of $\sym(5)$ and $\alt(5)$ are respectively $1$ and $2$. The \itbf{weak intersection density array} of an arc-transitive graph $X$ is the sub-array $\overline{\rho}(X)$ consisting of all $\rho\in \rho(X)$ that are realized by arc-transitive subgroups of automorphism.
	
	Cubic arc-transitive graphs have been extensively studied and they have been classified through their automorphism groups. For a comprehensive review of cubic arc-transitive graphs, the reader is referred to \cite[Chapter~10]{dobson2022symmetry}. A cubic arc-transitive graph $X = (V,E)$ is of \itbf{type $\{1,2^1\}$} if its full automorphism group $G^*$ is $2$-arc-regular, admitting a $1$-arc-regular subgroup $G$, and its edge stabilizer in $G^*$ is isomorphic to $\mathbb{Z}_2^2$. In \cite{kutnar2023intersection}, Kutnar, Maru\v{s}i\v{c} and Pujol initiated the study of the intersection densities of transitive automorphism groups of cubic arc-transitive graphs. In particular, they introduced the notion of weak intersection density array and considered it for this family of graphs.   
	By definition, cubic arc-transitive graphs of type $\{1,2^1\}$ have a 2-arc-regular full automorphism group, $G^*$ and a $1$-arc-regular transitive subgroup $G\leq G^*$. Hence, the weak intersection density array of such graph is $[\rho(G^*), \rho(G)]$ if $\rho(G)> \rho(G^*)$, and $[\rho(G)]$ if $\rho(G) = \rho(G^*)$.
	
	In~\cite{kutnar2023intersection}, under some strong restrictions, the weak intersection density array of many of these graph was computed. In this paper, we continue this work by determining the weak intersection density array of the cubic arc-transitive graphs of type $\{1,2^1\}$ with $G^* = \pgl{2}{q}$ and $G = \psl{2}{q}$, where $q$ is an odd prime power. We will assume the following hypothesis for convenience.

\begin{hyp}
	Let $X = (V,E)$ be an arc-transitive graph of type $\{1,2^1\}$ with full automorphism group $G^*=\pgl{2}{q}$, where $q=p^k$  is an odd prime power, and $G = \psl{2}{q}$ is a $1$-arc-regular subgroup of $ G^*$. \label{hyp}
\end{hyp}

If $X, G$ and $G^*$ are as in Hypothesis~\ref{hyp}, then for $v \in V$, there exists $h\in G$ of order $3$ and an involution $\nv \in G^*\setminus G$ such that $\nv h \nv^{-1} = h^{-1}$, $G_v = \langle h\rangle$ and $G_v^* = \langle h, \nv \rangle \cong\dih{6} \cong \sym(3)$. Under the action of $G^*=\pgl{2}{q}$ on $X$, exactly the elements conjugate to an element in $G_v^*\cong \sym(3)$ have a fixed point; in particular $g,h \in G^*$, are intersecting exactly when $hg^{-1}$ is conjugate to an element in $G^*_v$. Similarly, with this action, the elements with a fixed point in $G$ are exactly those conjugate to an element in $ \langle h\rangle$.

Whenever $q =2^k$, then $G = G^*$, so we do not consider these cases, and we assume henceforth that $p$ is an odd prime. 

As $|G| = \frac{(q-1)q(q+1)}{2}$ and $|G^*| = (q-1)q(q+1)$, it is clear that $G$ and $G^*$ always have elements of order $2$ and $3$. In fact, unless $q = 9$, $G = \psl{2}{q}$ is generated by an involution and an element of order $3$. Furthermore, $G^*$ is always an extension of $G$ by a cyclic group of order $2$, whenever $q$ is odd.

For $q$ a power of $3$, the following result was proved in~\cite[Theorem~7.2]{hujdurovic2022intersection-cyclic} and also in~\cite[Theorem 4.8]{kutnar2023intersection}.
\begin{thm}
	If $G = \psl{2}{3^k}$  and $G^*= \pgl{2}{3^k}$ with $k\geq 3$ an integer, is as in Hypothesis~\ref{hyp}, then 
	\begin{align*}
		\rho(G^*) = \rho(G) = 
		\begin{cases}
			3^{k-1} &\mbox{ if $k$ is odd}\\
			3^{\frac{k}{2}-1} &\mbox{ otherwise}.
		\end{cases}
	\end{align*}
	\label{thm:q=3}
\end{thm}

The next result was proved in \cite[Theorem~6.1]{hujdurovic2022intersection-cyclic} for the case when $q\equiv 1\pmod 3$.
\begin{thm}
	If $q=p^k\equiv 1\pmod 3$, and $G = \psl{2}{q}$ is as in Hypothesis~\ref{hyp}, then 
	\begin{align*}
		\rho(G) = 
		\begin{cases}
			\frac{4}{3} & \mbox{ if $p\neq 5$}\\
			2 &\mbox{ if $p=5$}.
		\end{cases}
	\end{align*}\label{thm:q=1mod3}
\end{thm}

Following Theorem~\ref{thm:q=3} and Theorem~\ref{thm:q=1mod3}, the only open cases for both $G=\psl{2}{q}$ and $G^*=\pgl{2}{q}$ are when $q = p^k\equiv 2\pmod 3$; in this paper we resolve these cases. Note if $q = p^k\equiv 2\pmod 3$, then $k$ is odd and $p\equiv 2 \pmod 3$. Theorem~\ref{thm:q=1mod3} computes the intersection density of arc-regular subgroups in Hypothesis~\ref{hyp}, but not the full automorphism groups. So for $q \equiv 1 \pmod 3$, we will also determine the intersection density of $G^*=\pgl{2}{q}$.

The main results of this paper, stated together with the known results, is the following.
	\begin{thm}
		If $X$ is a cubic arc-transitive graph as in Hypothesis~\ref{hyp}, with $q = p^k$ an odd prime power, and $(G,G^*) = \left(\psl{2}{q},\pgl{2}{q}\right)$, then 
		\begin{align*}
			\overline{\rho}(X)
			=
			\begin{cases}
				 [ 3^{k-1} ] &\mbox{ if $q=3^k$ and $k$ odd,}\\
			         [ 3^{\frac{k}{2}-1} ] & \mbox{ if $q=3^k$ and  $k$ even,} \\
				[1,2] & \mbox{ if $q\equiv 1\pmod 3$ and $p= 5$},\\
				[1,\tfrac{4}{3}] & \mbox{ if $q\equiv 1\pmod 3$ and $p\neq 5$},\\
				[1] & \mbox{ if $q\equiv 2\pmod 3$ and $q\equiv \pm 2\pmod 5$},\\
				[1,\tfrac{4}{3}] & \mbox{ if $q\equiv 2\pmod 3$ and $q \equiv \pm 1\pmod 5$ or   $q\equiv 0\pmod 5$}. 
			\end{cases}
		\end{align*}
		\label{thm:main}
	\end{thm}
	
\section{Background results}

\subsection{Orbital graphs}
	Throughout this section, we let $G\leq \sym(V)$ be a finite transitive group. The group $G$ acts intransitively on $V\times V$, and an orbit for this action is called an \itbf{orbital} of $G$. By transitivity of $G$, the set $O_0 = \{ (v, v): v\in V \}$ is an orbital, called the trivial orbital, or the diagonal orbital of $G$. Any other orbital of $G\leq \sym(V)$ is a subset of $\{ (u, v): u, v \in V,  \, u\neq v \}$. 
	
	Let $O$ be a non-trivial orbital of $G$. The set $O^* = \left\{ (v, u) : (u, v)\in O \right\}$ is also an orbital of $G$, and if $O$ is an orbital such that $O = O^*$, then $O$ is called a \itbf{symmetric orbital}. The orbital $O$ determines a digraph $X_O = (V,O)$, called an \itbf{orbital graph}, whose vertex set is $V$, and whose arc set is $O$. If $O$ is a symmetric orbital, then $X_O$ may be viewed as an undirected graph, otherwise, $X_O$ is an oriented digraph. A \itbf{basic orbital graph} of $G$ is an undirected graph of the form $X_{O}$ for some symmetric orbital $O$, or $X_{O}\cup X_{O^*}$ for some non-symmetric orbital $O$ of $G$.
	
	Recall that a graph $X=(V,E)$ is called \itbf{vertex transitive} if $\Aut(X)$ acts transitively on $V$. The next result shows that any vertex-transitive graph can be reconstructed from orbital digraphs.
	\begin{lem}\cite[Proposition~1.4.6]{dobson2022symmetry}
		Let $X = (V,E)$ be a vertex-transitive graph. If $G\leq \Aut(X)$ is transitive, then $X$ is a union of basic orbital graphs of $G$.\label{lem:transitive}
	\end{lem}
	
	A vertex-transitive graph $X= (V,E)$ is called a \itbf{Cayley graph} if there exists $H\leq \Aut(X)$ acting regularly on $V$, that is, for any $u, v \in V$, there exists a unique $h\in H$ such that $v = h(u)$. If $X$ is a Cayley graph, then $X$ is isomorphic to a graph $\operatorname{Cay}(H,C)$, where $C$ is an inverse-closed subset of $H\setminus \{1\}$, with vertex set $H$ and two elements $h$ and $h^\prime$ in $H$ are adjacent if and only if $h^\prime h^{-1} \in C$.
	
	Next, we recall an important correspondence between orbital graphs and the suborbits of a transitive subgroup of automorphisms. Recall that the \itbf{suborbits} of $G\leq \sym(V)$ are the orbits of $G_v$, for some $v\in V$. We will assume henceforth that $O_0,O_1,O_2,\ldots,O_d$ are the orbitals of $G\leq\ \sym(V)$, where $O_0$ is the trivial orbital. Fix $v\in V$. For any $0\leq i\leq d$, there exists $w_i\in V$ such that $O_i$ is equal to the orbital of $G$ containing $(v,w_i)$. By transitivity of $G$ on $O_i$, it is clear that $G_v$ is also transitive on the set $\Delta_i=\{w\in V:(v,w)\in O_i\}$. Therefore, there is a one-to-one correspondence between the orbitals of $G$ and the orbits of $G_v$, that is, the suborbits of $G$ with respect to $v$. In particular, for any $0\leq i\leq d$ the orbital $O_i$ of $G$ that contains $(v,w_i)$, corresponds to the suborbit $\Delta_i$ containing $w_i$. We say that a suborbit of $G$ is symmetric if the corresponding orbital is symmetric, otherwise it is called non-symmetric. Hence, we also define the graph $X_{\Delta_i} := X_{O_i}$ for $1\leq i\leq d$.

	The following result gives a characterization of arc-transitive digraphs.
	\begin{lem}\cite[Theorem~3.2.8]{dobson2022symmetry}
		Let $X=(V,E)$ be a graph, $G\leq \Aut(X)$ be transitive and $v\in V$. The following statements are equivalent.
		\begin{enumerate}[(i)]
			\item $X$ is an arc-transitive graph,
			\item the neighbourhood $N_X(v)$ of $v$ in $X$ is a symmetric suborbit of $G$,
			\item $X$ is an orbital graph of $G$.
		\end{enumerate}\label{lem:suborbits}
	\end{lem}
	
	\subsection{Cubic arc-transitive graphs}
	Let $X = (V,E)$ be a cubic arc-transitive graph of type $\{1,2^1\}$ with $G^* =\Aut(X)$. By Lemma~\ref{lem:suborbits}, the graph $X$ is isomorphic to an orbital graph $X_\Delta$ of any transitive subgroup of $G^*$, where $\Delta$ is a symmetric suborbit of size $3$. Fix $v\in V$ and let $\Delta = N_X(v) = \{u_1,u_2,u_3\}$. Since $G$ is a transitive subgroup of $G^*$, we know that $G_v$ acts transitively on $\Delta$. Using the fact that $G$ is $1$-arc regular, the subgroup $G_v$ acts regularly on the set of arcs $\{ (v, u_1),(v, u_2),(v, u_3) \}$. Therefore, $G_v$ must be a cyclic group of order $3$. We assume that $G_v = \langle h\rangle$, for some $h\in G$ of order $3$. By the orbit-stabilizer lemma, we know that $|V| = \frac{|G|}{|G_v|} = \frac{q(q^2-1)}{2 (3)}$.
	As $\Delta$ is also a suborbit of $G^*$, the orbit-stabilizer lemma implies that 
	$|G^*_v| = \frac{|G^*|}{|V|} = 6$. 
	Since $G_v^*$ is transitive on $\Delta$, there exists an involution $\nv \in G_v^*$ such that $\nv h \nv = h^{-1}$, and $G_v^* = \langle  h, \nv \rangle$.

\subsection{Derangement graphs}
	
	Let $G\leq \sym(V)$ be a transitive group. The derangement graph is an important tool to determine the intersection density of transitive groups. The \itbf{derangement graph} $\Gamma_{G}$ is the graph with vertex set $G$, and two elements $g$ and $h$ are adjacent if $hg^{-1}$ is a derangement (that is, a fixed-point-free permutation). An important property of the stabilizers of $G$ is that if $u,v\in V$ and $g\in G$ such that $g(u)=v$, then $G_u = g^{-1}G_vg$. Therefore, $g$ and $h$ are adjacent in $\Gamma_{G}$ if and only if $hg^{-1}$ is not conjugate to any element in $G_v$, for any fixed $v\in V$.
	
	If $D_G$ is the set of derangements of $G$, then it is not hard to see that $\Gamma_{G} = \operatorname{Cay}(G,D_G)$. The derangement graph $\Gamma_{G}$ is defined in a way that $\mathcal{F}$ is intersecting with respect to the action of $G$ on $V$ if and only if $\mathcal{F}$ is a coclique of $\Gamma_{G}$. Therefore, one can extend the definition of the intersection density for $G$ acting on $V$ as follows
	\begin{align*}
		\rho(G) = \frac{\alpha(\Gamma_{G})}{|G|/|V|}.
	\end{align*}

\section{Properties of the linear groups $\psl{2}{q}$ and $\pgl{2}{q}$}
\label{sect:maximal-subgroups}
	
In this section, we determine the properties of $\psl{2}{q}$ and $\pgl{2}{q}$ that we will need. This includes the structure of the cyclic subgroups of $\psl{2}{q}$. Along the way, we will give the normalizer and the point-stabilizers of these elements.  We start with some remarks about the elements of $\psl{2}{q}$ and $\pgl{2}{q}$. We will assume $q$ is an odd prime power. 
	
	\subsection{Properties of elements}
	Any element of $\psl{2}{q}$ is conjugate to an element of a certain cyclic subgroup. In particular, if $g\in \psl{2}{q}$ is of order $k>2$, then $g$ is conjugate to an element of a subgroup isomorphic to: $\mathbb{Z}_{\frac{q-1}{2}}$, if $k \mid \frac{q-1}{2}$; $\mathbb{Z}_{\frac{q+1}{2}}$, if $k\mid \frac{q+1}{2}$; and $\mathbb{Z}_{p}$, if $k\mid q$. An involution of $\psl{2}{q}$ is conjugate to an element in  $\mathbb{Z}_{\tfrac{q-1}{2}}$ if $q\equiv 1 \pmod 4$, and to an element in $\mathbb{Z}_{\tfrac{q+1}{2}}$ if $q\equiv 3 \pmod 4$. The normalizers of the cyclic subgroups of $\psl{2}{q}$ are given in the next table; the table for $\pgl{2}{q}$ is given following some comments.
	
	\begin{table}[H]
		\centering
		\begin{tabular}{|c|c|c|c|c|}
			\hline
			Congruence of $q$& $o(g) = 2$& $o(g) \mid \frac{q-1}{2}$ & $o(g) \mid \frac{q+1}{2}$& $o(g) \mid q$\\ \hline \hline
			$q\equiv 1\pmod 4$& $\dih{q-1}$& $\dih{q-1}$ & $\dih{q+1}$& $\mathbb{Z}_p^k: \mathbb{Z}_{\frac{p-1}{2}}$\\ \hline
			$q\equiv 3\pmod 4$& $\dih{q+1}$& $\dih{q-1}$ & $\dih{q+1}$& $\mathbb{Z}_p^k: \mathbb{Z}_{\frac{p-1}{2}}$\\ \hline
		\end{tabular}
		\caption{The normalizer $\operatorname{N}_G(\langle g\rangle)$ of $\langle g\rangle$ in $G = \psl{2}{q}$, when $q$ is an odd prime power.}\label{tab:normalizer-psl}
	\end{table}
	
	Similar to the case for $\psl{2}{q}$, an element $g \in \pgl{2}{q}$ of order $k>2$ is conjugate to an element in a cyclic group isomorphic to: $\mathbb{Z}_{q-1}$, if $k \mid (q-1)$; $\mathbb{Z}_{q+1}$, if $k\mid (q+1)$; and $\mathbb{Z}_{p}$ if $k\mid q$. One main difference between $\psl{2}{q}$ and $\pgl{2}{q}$ is the structure of involutions. For $\psl{2}{q}$ there is a unique conjugacy class of involutions, these involutions are contained in either $D_{q-1}$ or $D_{q+1}$, depending on the congruence of $q$ modulo $4$. For $\pgl{2}{q}$, there are two conjugacy classes of involutions if $q$ is odd (but only one if $q$ is even). For $\pgl{2}{q}$, the normalizer of cyclic subgroups are given below. 
	\begin{table}[H]
		\centering
		\begin{tabular}{|c|c|c|c|c|c|}
			\hline
	Congruence & $o(g) = 2$, & $o(g) = 2$  & $o(g) \mid (q-1)$ & $o(g) \mid (q+1)$ & $o(g) = p$\\ 
	 of $q$   &  $g \in \psl{2}{q}$ &  $g \not \in \psl{2}{q}$ &   &  &  \\ \hline \hline
	$q\equiv 1\pmod 4$ & $\dih{2(q-1)}$ & $\dih{2(q+1)}$ & $\dih{2(q-1)}$ & $\dih{2(q+1)}$  & $\mathbb{Z}_p^k:\mathbb{Z}_{p-1}$ \\ \hline
	$q\equiv 3\pmod 4$ & $\dih{2(q+1)}$ & $\dih{2(q-1)}$ & $\dih{2(q-1)}$ & $\dih{2(q+1)}$  & $\mathbb{Z}_p^k:\mathbb{Z}_{p-1}$ \\ \hline
		\end{tabular}
		\caption{The normalizer $\operatorname{N}_{G^*}(\langle g\rangle)$ of $\langle g\rangle$ in $G^* = \pgl{2}{q}$, when $q$ is an odd prime power.}\label{tab:normalizer-pgl}
	\end{table}

\subsection{Dihedral subgroups of order $6$} 
	
	As the point-stabilizers of the full automorphism groups that we study in this paper are isomorphic to $\dih{6}\cong \sym(3)$, we need to recall some facts about the conjugacy classes of these subgroups in $\pgl{2}{q}$.
	
	\begin{lem}\cite{cameron20063}
		Let $q = p^k$ be odd. The number of conjugacy classes of subgroups of $\pgl{2}{q}$ isomorphic to the dihedral group $\dih{6}$ are as follows.
		\begin{enumerate}[(i)]
			\item If $p\neq 3$, then there are two conjugacy classes of subgroups isomorphic to $\dih{6}$. In particular, one class lies in $\psl{2}{q}$ and the other does not.
			\item If $q=3^k$, then there is a unique conjugacy class of subgroups isomorphic to $\dih{6}$. If $k$ is even, this class lies in $\psl{2}{q}$, and if $k$ is odd, then it does not.
		\end{enumerate}\label{lem:conj-class-3}
	\end{lem}
	
		If $q$ is odd and not a power of $3$, then there are two conjugacy classes of subgroups isomorphic to $\dih{6} \cong \sym(3)$, so we let $H\leq \psl{2}{q}$ and $H^\prime \not \leq \psl{2}{q}$ be representatives of these two conjugacy classes. Since neither of $H$ or $H^\prime$ contain $\psl{2}{q}$, $H$ and $H^\prime$ are core-free subgroups. Consequently, the actions of $\pgl{2}{q}$ on cosets of $H$ and $H^\prime$ by multiplication are faithful, and thus correspond to transitive permutation groups.  
	
	Let $V^\prime$ be the set of cosets of $H^\prime$ in $\pgl{2}{q}$, we claim that the action of $\psl{2}{q}$ on $V^\prime$ is transitive. To see this, let $\nv \in \pgl{2}{q}\setminus \psl{2}{q}$ be an involution. Then, $\pgl{2}{q} = \psl{2}{q}\rtimes \langle \nv \rangle$. If $xH^\prime$ and $yH^\prime$ are two cosets in $V^\prime$, then there exist $x^\prime, y^\prime\in \psl{2}{q}$ such that $x = x^\prime \nv$ and $y  = y^\prime \nv$. Then, we have
	\begin{align*}
		(y^\prime  (x^\prime)^{-1}) (xH^\prime) &= (y^\prime  (x^\prime)^{-1}) (x^\prime \nv H^\prime)=y^\prime \nv H^\prime = yH^\prime .
	\end{align*}
	
	Let $V$ be the set of cosets of $H$ in $\pgl{2}{q}$, we claim that the action of $\psl{2}{q}$ on $V$ is intransitive. Fix $v\in V$ and let $g\in \pgl{2}{q}\setminus \psl{2}{q}$ such that $u = g(v)$. Then assume that there is a $g^\prime \in \psl{2}{q}$, such that $u = g^\prime (v)$. This implies $g^{-1}g^\prime (v) = v$, and $g^{-1}g^\prime$ is conjugate to an element in $H$, so conjugate to an element in $\psl{2}{q}$. But since $g^{-1}g^\prime \in \pgl{2}{q}\setminus \psl{2}{q}$, this is not possible and no such $g^\prime$ exists in $\psl{2}{q}$. 
	
	Consequently, we conclude that if $q$ is not a power of 3 and $\psl{2}{q}$ acts transitively on the cosets of a subgroup of $\pgl{2}{q}$ isomorphic to $\dih{6}$, then the subgroup is conjugate to $H^\prime$. In particular, the point stabilizers of this action are conjugate to $H^\prime$.
			
	Next consider $q=3^k$ for some integer $k\geq 1$, by Lemma~\ref{lem:conj-class-3} in this case there is a unique conjugacy class of subgroups of isomorphic to $\sym(3)$. If $k$ is even, then a copy of $\sym(3)$ given by the normalizer of an element of order $3$ in $\psl{2}{q}$ lies in $\psl{2}{q}$. Consequently, using the same argument before, $\psl{2}{q}$ cannot be transitive on $V$ in this case. If $k$ is odd, no subgroup of $\pgl{2}{q}$ isomorphic to $\sym(3)$ lies in $\psl{2}{q}$. Similar to the case where $q$ is not a power of $3$, the subgroup $\psl{2}{q}$ is transitive.
	
	We summarize what we showed in this section, using the fact that Hypothesis~\ref{hyp} implies the action of $\psl{2}{q}$ is transitive.
	
	\begin{lem}
		Let $X$ be a cubic arc-transitive graph of type $\{1,2^1\}$ satisfying Hypothesis~\ref{hyp} with $G^* = \pgl{2}{q}$ and $G = \psl{2}{q}$, then no vertex-stabilizer of $\pgl{2}{q}$ can lie in $\psl{2}{q}$. In particular, if $q = 3^k$, then $k$ is odd.\label{lem:psl-transitive}
	\end{lem}
	
	\subsection{The elements of order $3$}
	Let us determine the elements of order $3$ of $\psl{2}{q}$, where $q = p^k$ for some odd prime number $p$ and some integer $k\geq 1$. Since $\psl{2}{q} = \sln{2}{q}/\{\pm I\}$, an element of order $3$ of $\psl{2}{q}$ is of the form $\overline{A}$ such that $A\in \sln{2}{q}$ and $A^3=\pm I$.
	Let $\chi(t)$ be the characteristic polynomial of $A$. 
	
	First assume that $p\neq 3$. Since $A^3 = \pm I$ we have 2 cases. When $A^3 = I$, the polynomial $t^3-1 = (t-1)(t^2+t+1)$ is annihilated by $A$. Hence, the minimal polynomial of $A$ divides $(t-1)(t^2+t+1)$. It is clear that the minimal polynomial cannot be of degree $1$, otherwise $A$ would be the identity. Since $A \neq I$, its characteristic polynomial, $\chi(t)$, which has degree 2 is therefore equal to its minimal polynomial. As $A\in \sln{2}{q}$, the constant term of $\chi(t)$ is equal to $1$, and  therefore $\chi(t) = t^2+t+1$ and so $A$ is similar to the matrix
	\begin{align*}
		\begin{bmatrix}
			0 & -1\\
			1 & -1
		\end{bmatrix}.
	\end{align*} 
	Consequently, $\operatorname{Tr}(A) = -1$ and $\det(A) = 1$. Next consider the case where $A^3 = -I$, then $t^6-1 = (t-1)(t+1)(t^2+t+1)(t^2-t+1)$ is annihilated by $A$. Again, the minimal polynomial of $A$ cannot have degree $1$, so it must coincide with $\chi(t)$. Since $A\in \sln{2}{q}$, the constant term of $\chi(t)$ is equal to $1$, and we must have that its characteristic polynomial $\chi(t)\in \{ t^2-1,t^2+t+1,t^2-t+1 \}$. Clearly, $\chi(t)\neq t^2-1$ since the order of $A$ is $3$. We conclude that $\chi(t) = t^2+t+1$ or $\chi(t) = t^2-t+1$, in either case we have $\operatorname{Tr}(A) = \pm 1.$ 
	
	Now, assume that $p = 3$, and again we consider two cases. If $A^3 = I$, then $t^3-1 = (t-1)^3$ is annihilated by $A$. This implies $\chi(t) = (t-1)^2 = t^2+t+1$. If $A^3 = -I$, then $t^3+1 = (t+1)^3$ is annihilated by $A$, so $\chi(t) = (t+1)^2 = t^2-t+1$. In either case, we conclude that $\operatorname{Tr}(A) = \pm 1$.
	
	Hence, we have shown that any element $\overline{A} \in \psl{2}{q}$ of order $3$ must have the property that $\operatorname{Tr}(A) = \pm 1$. It is an easy exercise to show that the only representative of conjugacy classes of $\sln{2}{q}$ with traces $\pm 1$ are elements $A$ such that $A^3 = \pm I$. Therefore, the converse of the earlier result also holds. We summarize these results in the next lemma.
	\begin{lem}
		An element $\overline{A} \in \psl{2}{q}$ has order $3$ if and only if $\operatorname{Tr}(A) = \pm 1$.
	\end{lem}
	
	\subsection{The elements of order $2$ in $\pgl{2}{q}$}

	Next we determine the conjugacy classes of involutions of $\pgl{2}{q}$ not belonging to $\psl{2}{q}$. 

	\begin{lem}\label{lem:trace0}
		If $\overline{A} \in \pgl{2}{q} \setminus \psl{2}{q}$ is an involution, then $\operatorname{Tr}(A) = 0$.\label{lem:trace-involution}
	\end{lem}
	\begin{proof}
		We note that $\pgl{2}{q}$ has two conjugacy classes of involutions. The first class consists of the involutions in $\psl{2}{q}$ and the second one  those in $\pgl{2}{q} \setminus \psl{2}{q}$. 
		Every involution in $\pgl{2}{q} \setminus \psl{2}{q}$ is conjugate to the element
		\begin{align*}
			\overline{\begin{bmatrix} 				0 & 1\\ 				1 & 0 		\end{bmatrix}}.
		\end{align*}
		The result follows immediately by noting that the trace of the above element is $0$.
	\end{proof}

	\subsection{Transversals} 	
	
 In this section we build a subgroup of order $q+1$ in $\pgl{2}{q}$ and determine a transversal for this subgroup. To do this, consider an element $h$ with order 3 in $\pgl{2}{q}$, we can assume without loss of generality that
	\begin{align*}
		h =
		\overline{ \begin{bmatrix}
			0 & -1\\
			1 & -1
		\end{bmatrix} }.
	\end{align*}
Clearly, $\pgl{2}{q}$ acts transitively on the conjugacy class that includes $h$, acting by conjugation. The point-stabilizer of this action is $\operatorname{C}_{\pgl{2}{q}}(h)$. The centralizer of $h$, as matrix in $\gl{2}{q}$, is the set of matrices 
$ x = \begin{bmatrix} a & b\\ c & d \end{bmatrix}$ with 
	\begin{align*}
		\begin{bmatrix}
			0 & -1\\
			1 & -1
		\end{bmatrix}
		\begin{bmatrix}
			a & b\\
			c & d
		\end{bmatrix}
		=
		\begin{bmatrix}
			a & b\\
			c & d
		\end{bmatrix}
		\begin{bmatrix}
			0 & -1\\
			1 & -1
		\end{bmatrix},
	\end{align*}
	or, equivalently, with
	\begin{align*}
		\begin{bmatrix}
			-c & -d\\
			a-c & b-d
		\end{bmatrix}
		=
		\begin{bmatrix}
			b & -a-b\\
			d & -c-d
		\end{bmatrix}.
	\end{align*}
	Consequently, $b = -c$ and $d = a+b$. From these equations, we have the centralizer of $h$ in $\pgl{2}{q}$ is
	\begin{align*}
		\overline{x} = \overline{ \begin{bmatrix}
			a & b\\
			c & d
		\end{bmatrix}}
		=
		\overline{ \begin{bmatrix}
			a & b\\
			-b & a+b
		\end{bmatrix}}
		=
		\begin{cases}
			 \overline{  \begin{bmatrix}
				1 & \alpha\\
				-\alpha & 1+\alpha
			\end{bmatrix} }
			&
			\mbox{ for some $\alpha \in \mathbb{F}_{q}$ if }a\neq 0,\\[15pt]
			 \overline{  \begin{bmatrix}
				0 & 1\\
				-1 & 1
			\end{bmatrix} }
			&
			 a= 0 .
		\end{cases}
	\end{align*}
	We therefore conclude that the centralizer of $h$ in $\pgl{2}{q}$ is the subgroup
	\begin{align}
		S = \operatorname{C}_{\pgl{2}{q}}(h)
		= 
		\left\{
		 \overline{  \begin{bmatrix}
			1 & \alpha\\
			-\alpha & 1+\alpha
		\end{bmatrix} } 
		: 
		\alpha \in \mathbb{F}_{q}
		\right\}\cup \left\{ 
		     \overline{  \begin{bmatrix}
			0 & 1\\
			-1 & 1
		\end{bmatrix} }
		 \right\}\cong \mathbb{Z}_{q+1}.\label{eq:centralizer}
	\end{align}

In Lemma~\ref{lem:main} we will determine the suborbits of $S$ acting by conjugation on the conjugacy class that contains $h$.  In order to do this, we first need to identify some representatives of cosets of $S$. Define the subgroup 
	\begin{align}
		K = \left\{
		 \overline{  \begin{bmatrix}
			1 & a\\
			0 & b
		\end{bmatrix} }
		:
		a\in \mathbb{F}_{q} \mbox{ and } b\in \mathbb{F}_{q}^* \right\} \cong \mathbb{F}_{q}\rtimes \mathbb{Z}_{q-1}.\label{eq:transversal}
	\end{align}
	\begin{prop}
		The subgroup $K$ is a left-transversal of $S$ in $\pgl{2}{q}$.\label{prop:transversal}
	\end{prop}
	\begin{proof}
		Note that $|S| = q+1$ and $|K|=q(q-1)$, so the elements in $S\cap K$ cannot have order more than $2$. As $2\mid (q+1)$ and $2\mid (q-1)$, $S$ and $K$ both admit involutions, however, the involutions in $S$ are conjugate to elements of $\psl{2}{q}$, whereas the involutions in $K$ cannot be conjugate to any element of $\psl{2}{q}$. Therefore, $|S\cap K| = 1$. Since $S\cap K$ is trivial and $K$ is a subgroup, if $x,y\in K$ such that $x^{-1}y\in S$, then $x = y$.
	\end{proof}
	
	From this proposition, we deduce that $\psl{2}{q}/S = \{ kS: k\in K \}$. Note that an element $kS$, with $k\in K$, can be identified with the element of order $3$ given by $khk^{-1}$.
		
	
  		\section{Outline of the proof of Theorem~\ref{thm:main}}
	
	Let $X = (V,E)$ be a cubic arc-transitive graph satisfying Hypothesis~\ref{hyp}. Then, we have that $G^* = \pgl{2}{q}$ and $G = \psl{2}{q}$, for some odd prime power $q$. Recall that $G_v = \langle h\rangle$ and $G_v^* = \langle h, \nv \rangle$. Given an intersecting set $\mathcal{F} \subset G$, we have seen that we may assume $1\in \mathcal{F}$. Therefore, every non-identity element of $\mathcal{F}$ is conjugate to $h$. If $|\mathcal{F}|\geq 2$, then it contains an element $x$ of order $3$, and there exists $g\in G$ such that $gxg^{-1} \in \{h,h^{-1}\}$. Therefore, we may also assume that $\mathcal{F}$ contains $h$ or $h^{-1}$, provided that $|\mathcal{F}|\geq 2$, as we can replace $\mathcal{F}$ with the set $g\mathcal{F}g^{-1}$. Therefore, we will always assume that 
	\begin{align}
		\mbox{\it any intersecting set $\mathcal{F} \subset G$ such that $|\mathcal{F}|\geq 2$ contains $1$ and $h$}.\label{eq:assumption}
	\end{align}
	
	We will now give a sketch of the proof of Theorem~\ref{thm:main}. The main result is the intersection density of $G$ when $q = p^k\equiv 2\pmod 3$. Note that since $q = p^k\equiv 2 \pmod 3$, $p \equiv 2 \pmod 3$ and $k$ is odd.
	
We will use the fact that all elements of order $3$ in $G^*$, that are conjugate to elements of $G_v  = \langle h\rangle$, must be contained in $G$. Let $\mathcal{C}_3$ be the set of all the elements of $G$ of order $3$, since $h$ and $h^{-1}$ are conjugate in $G$, it follows that $\mathcal{C}_3$ is a single conjugacy class in $G$.

By definition, in the derangement graph $\Gamma_{G}$ of $G$, two vertices $x, y \in G$ are not adjacent in $\Gamma_{G}$ if and only if $xy^{-1}$ is conjugate to an element in $G_v$.  So a clique in $\overline{\Gamma_G}$ is an intersecting set of permutations. Hence, $\mathcal{C}_3$ is the neighbourhood of the identity permutation in $\overline{\Gamma_{G}}$, and by~\eqref{eq:assumption}, any maximum clique of $\overline{\Gamma_{G}}$ is contained in $\{1\}\cup \mathcal{C}_3$, and, provided it is of size at least $2$, contains $h$. Now, let $\Gamma$ be the first subconstitutent of $\overline{\Gamma_{G}}$, that is, $\Gamma = \overline{\Gamma_{G}}[\mathcal{C}_3]$ is the subgraph of $\overline{ \Gamma_{G} }$ induced by $\mathcal{C}_3$. It is clear that
	\begin{align}
		\rho(G) = \frac{\omega(\overline{\Gamma_G})}{|G|/|V|} = \frac{1+\omega(\Gamma)}{|G|/|V|}. \label{eq:density}
	\end{align}
	The analysis now entirely depends on the first subconstituent of $\Gamma$. The main step of the proof is the next lemma that we prove in the following section.
	
	\begin{lem}
		Assume that $q=p^k$, where $p$ is odd and $p\neq 3$ and let $\Gamma = \overline{\Gamma_{G}}[\mathcal{C}_3]$ as above.
		\begin{enumerate}[(i)]
			\item $G^* \leq \Aut(\Gamma)$ and acts transitively by conjugation. The vertex stabilizer with this action is isomorphic to $\operatorname{C}_{G^*}(h)$.\label{lem:main-item-1}
			\item If $q\equiv 2\pmod 3$, then $\Gamma$ is either a perfect matching or the union of two orbital graphs of $G^*$ (these orbital graphs correspond to the suborbits $\{h^{-1}\}$ and a set $N$, with $|N| = q+1$). \label{lem:main-item-2}
			\item If $q\equiv 2\pmod 3$, then the subgraph of $\Gamma$ induced by $N$ is a Cayley graph on $\mathbb{Z}_{q+1}$. \label{lem:main-item-3}
		\end{enumerate}
		
		\label{lem:main}
	\end{lem}
	
		Since $\Gamma$ is vertex transitive, all first subconstitutents are isomorphic. Let $\Delta$ be the neighbourhood of $h$ in $\Gamma$. Note that $h^{-1} \in \Delta$. Define $\tilde{\Gamma} = \Gamma[\Delta\setminus \{h^{-1}\}]$. We will show that $\tilde{\Gamma}$ is either the empty graph, a union of cycles of length at least $4$, or a perfect matching. From this, we completely determine all possible values of $\omega(\Gamma)$, and therefore, $\rho(G)$.
	
	Then we consider the final cases for $\rho(G^*)$. By Lemma in \cite{meagher180triangles}, we know that $\rho(G^*)\leq \rho(G)$. If $\rho(G) = 1$, then it is clear that $\rho(G^*) = 1$. If $\rho(G)>1$, then we also show that $\rho(G^*) = 1$. To do this, we consider $\mathcal{F} \subset G^*$ an intersecting set of maximum size, then look at the structure of  $|\mathcal{F} \cap G|$.

	\section{Proof of Lemma~\ref{lem:main}}
	
	Let $X,G,G^*,p,$ and $k$ be as in Hypothesis~\ref{hyp}. Recall that $\mathcal{C}_3$ is the conjugacy class of order three elements, $\Gamma = \overline{\Gamma_{G}}[\mathcal{C}_3]$, and $\tilde{\Gamma} = \Gamma[\Delta]$, where $\Delta = N_{\Gamma}(h)$. Assume that $p\neq 3$, so $q \equiv \pm 1 \pmod 3$. 
	
	\subsection{Proof of Lemma~\ref{lem:main}\eqref{lem:main-item-1}} 
		
	Since the vertices of $\Gamma$ form a single conjugacy class in $G^*$, it is clear that $G^*$ acts on transitively on the vertices of $\Gamma$ by conjugation. To prove that $\Gamma$ is vertex transitive, it is enough to show that conjugation is an automorphism of $\Gamma$. Recall that two elements $g$ and $g^\prime$ of $\mathcal{C}_3$ are adjacent if and only if $g^\prime g^{-1} \in \mathcal{C}_3$. If $g,g^\prime \in \mathcal{C}_3$ and $x\in G^*$, then using the fact that $\mathcal{C}_3$ is a conjugacy class, we have
		\begin{align*}
				(xg^\prime x^{-1}) \ (x g^{-1} x^{-1}) \in \mathcal{C}_3 \Longleftrightarrow g^\prime g^{-1} \in \mathcal{C}_3.
			\end{align*}
			
	This proves that the transitive action of $G^*$ on the vertices of $\Gamma$ by conjugation preserves the edges and non-edges of $\Gamma$. If the kernel of this action is non-trivial, then there exists a non-trivial element $x\in G^*$ that commutes with every element of $\mathcal{C}_3$. As $p\neq 3$, every subgroup of order $3$ is self-centralizing, so $x$ belongs to every subgroup of order $3$ of $G$ and can only be the identity. This implies that the kernel of the action of $G^*$ is trivial, or equivalently, $G^*$ acts faithfully. 
	
	Hence, $G^*\leq \Aut(\Gamma)$, acting by conjugation, is a transitive subgroup, and thus $\Gamma$ is vertex transitive. The point-stabilizer of this transitive action is the subgroup $\operatorname{C}_{G^*}(h)$, which is a cyclic group of order $q+1$. By the orbit stabilizer lemma, the cosets of $\operatorname{C}_{G^*}(h)$ in $G$ are in one-to-one correspondence with elements of $\mathcal{C}_3$, so we can either think about cosets of $\operatorname{C}_{G^*}(h)$ in $G$, or elements of $\mathcal{C}_3$. In particular, the correspondence is given by
		\begin{align}
				ghg^{-1} \leftrightarrow g\operatorname{C}_{G^*}(h).\label{eq:correspondence}
		\end{align} This completes the proof of \eqref{lem:main-item-1}.

	\subsection{Proof of Lemma~\ref{lem:main}\eqref{lem:main-item-2}} 
	
	Assume $q \equiv 2 \pmod{3}$ and $q$ odd. By Lemma~\ref{lem:transitive}, $\Gamma$ is a union of basic orbital graphs of $G^*$.
	Let $\Delta = N_{\Gamma}(h)$. Then, $\Delta$ is a union of suborbits. As $h$ and $h^{-1}$ are adjacent in $\Gamma$ and $\{h^{-1}\}$ is an orbit of $\operatorname{C}_{G^*}(h)$, it is clear that $\{h^{-1}\}$ is among these suborbits. If $\Gamma$ is equal to the orbital graph corresponding to the suborbit $\{h^{-1}\}$, then it is a perfect matching. So we may assume that $\Gamma$ is not a perfect matching, implying that $\Delta \setminus \{h^{-1}\}\neq \varnothing$.	 We define the set $N := \Delta \setminus \{h^{-1}\}$. We claim that $N$ is an orbit of $C_{G^*}(h)$ of size $q+1$.	 
	  
	First we prove that $|N|\geq q+1$. As $p\neq 3$, every subgroup of order $3$ is self-centralizing, so for any $h^\prime\in \mathcal{C}_3$ such that $h^\prime \not \in \langle h\rangle$, we have $\operatorname{C}_{G^*}(h) \cap \operatorname{C}_{G^*}(h^\prime) = \{1\}$. Let $x \in N$. If $\left|\left\{gxg^{-1}: g\in \operatorname{C}_{G^*}(h)\right\}\right| < q+1$, then there exist $g,g^\prime \in\operatorname{C}_{G^*}(h)$ such that $gxg^{-1} = g^{\prime}x(g^\prime)^{-1}$. But, this implies $g^{-1}g^\prime \in \operatorname{C}_{G^*}(x)$, since we have $g,g^\prime \in \operatorname{C}_{G^*}(h)$ and $\operatorname{C}_{G^*}(h) \cap \operatorname{C}_{G^*}(x) = \{1\}$, we know that $g = g^\prime$. Consequently, $\left|\left\{gxg^{-1}: g\in \operatorname{C}_{G^*}(h)\right\}\right| = q+1$, and so $|N|\geq q+1$.
	
	Next, we prove that $| N |\leq q+1$. As $q\equiv 2\pmod 3$, there is a unique conjugacy class of subgroups of order $3$ in $G$. We set $H = \langle h \rangle$ with
	\begin{align*}
		h =
		\overline{\begin{bmatrix}
				0 & -1\\
				1 & -1
		\end{bmatrix}}.
	\end{align*}
	Recall that the subgroup $K$ defined in \eqref{eq:transversal} is a transversal of $S = \operatorname{C}_{G^*}(h)$. In \eqref{eq:correspondence}, the element $x = ghg^{-1}\in N$ corresponds to the coset $g\operatorname{C}_{G^*}(h)$. By the fact that $K$ is a transversal of $\operatorname{C}_{G^*}(h)$ in $G^*$, we may choose $g=k$ to be in $ K$ so that $x = ghg^{-1}=khk^{-1} \in N $.
	Hence, we consider the non-identity element
	\begin{align*}
		k = \overline{\begin{bmatrix}
			1 & a\\
			0 & b
		\end{bmatrix}} \in K,
	\end{align*}
	that is, $(a,b)\neq (0,1)$. Now, we have
	\begin{align}\label{eq:conjofh}
		\overline{\begin{bmatrix}
			1 & a\\
			0 & b
		\end{bmatrix}}
		\overline{\begin{bmatrix}
			0 & -1\\
			1 & -1
		\end{bmatrix}}
		\overline{\begin{bmatrix}
			1 & a\\
			0 & b
		\end{bmatrix}}^{-1}
		=
		\overline{\begin{bmatrix}
			a & -b^{-1}(a^2+a+1)\\
			b & -1-a
		\end{bmatrix}}.
	\end{align}
	In addition, we have
	\begin{align}
		\overline{
		\begin{bmatrix}
			1 & a\\
			0 & b
		\end{bmatrix}
		}
		\overline{
		\begin{bmatrix}
			0 & -1\\
			1 & -1
		\end{bmatrix}
		}
		\overline{
		\begin{bmatrix}
			1 & a\\
			0 & b
		\end{bmatrix}}^{-1}
		\overline{
		\begin{bmatrix}
			0 & -1\\
			1 & -1
		\end{bmatrix}}^{-1}
		=
		\overline{
		\begin{bmatrix}
			-a + b^{-1}\left(a^{2} + a + 1\right) & a \\
			a - b + 1 & b
		\end{bmatrix}}.\label{eq:trace}
	\end{align}
	The above expressions will be useful due to the fact that $ghg^{-1}$ and $h$ are adjacent if and only if $ghg^{-1}h^{-1}$ is an element of order $3$.  By \eqref{eq:trace}, 
	\begin{align*}
		\overline{\begin{bmatrix}
			1 & a\\
			0 & b
		\end{bmatrix}}
		\overline{\begin{bmatrix}
			0 & -1\\
			1 & -1
		\end{bmatrix}}
		\overline{\begin{bmatrix}
			1 & a\\
			0 & b
		\end{bmatrix}}^{-1}
		\mbox{ and }
		\overline{\begin{bmatrix}
			0 & -1\\
			1 & -1
		\end{bmatrix}}
	\end{align*}
	are adjacent if and only if $b-a + b^{-1}\left(a^{2} + a + 1\right) = \pm 1.$ We can reformulate this as
	\begin{align}
		b^2-ab+a^2+a+1 = \pm b. \label{eq:trace-1}
	\end{align}
	
	If $b^2-ab+a^2+a+1 = -b$, then we have $\left(b+(1-a)2^{-1}\right)^2-(1-a)^24^{-1}+a^2+a+1 = 0.$ Hence, if $z = b+(1-a)2^{-1}$, then 
	\begin{align}
		z^2 + (1-4^{-1})a^2+(2^{-1}+1)a+(1-4^{-1}) = 0.\label{eq:-1}
	\end{align}
	By letting $X = 2z$ and $Y = a+1$, we can see that \eqref{eq:-1} becomes
	\begin{align}
		X^2+3Y^2 = 0.\label{eq:-1-final}
	\end{align}
	As $-3$ is not a square whenever $q\equiv 2\pmod 3$, we conclude that \eqref{eq:-1-final} has a unique solution, given by $(X,Y) = (0,0)$, or equivalently, $a = b = -1$. In particular, we have
	\begin{align*}
		\overline{\begin{bmatrix}
			1 & -1\\
			0 & -1
		\end{bmatrix}}
		\overline{\begin{bmatrix}
			0 & -1\\
			1 & -1
		\end{bmatrix}}
		\overline{\begin{bmatrix}
			1 & -1\\
			0 & -1
		\end{bmatrix}}^{-1}
		=
		\overline{\begin{bmatrix}
			-1 & 1\\
			-1 & 0
		\end{bmatrix}}.
	\end{align*}  
	Therefore, this solution corresponds to $h^{-1}$.
	
	If $b^2-ab+a^2+a+1 = b$, then $b^2-(a+1)b+a^2+a+1 = \left(b-(a+1)2^{-1}\right)^2-4^{-1}(a+1)^2+a^2+a+1 =  0$. Setting $z = b-(a+1)2^{-1}$, this is 
	\begin{align}
		z^2 + (1-4^{-1})a^2+(1-2^{-1})a+(1-4^{-1}) = 0,
	\end{align}
which is equivalent to 
	\begin{align}
		4 z^2 + 3 a^2+ 2 a+3 = 0.\label{eq:+1}
	\end{align}
	Similar to the previous case, we set $X = 2z$, $Y = a+ 3^{-1}$, and $\gamma = 1- 3^{-2}$ and \eqref{eq:+1} becomes 
	\begin{align}
		X^2+3Y^2 + 3\gamma = 0.\label{eq:+1-final}
	\end{align}
	From \cite[Lemma~6.24]{lidl1997finite}, the above equation has exactly $q+1$ solutions in $\mathbb{F}_q$, as $-3$ is not a square in $\mathbb{F}_q$. Each of these $q+1$ solutions produce a unique element of $N$. Therefore, $|N| = q+1$. This completes the proof.
		
	\subsection{Proof of Lemma~\ref{lem:main}~\eqref{lem:main-item-3}}
	We saw in the previous section that if $q\equiv 2 \pmod 3$, then $\Gamma$ is a union of at most two orbital graphs. In particular, the neighbourhood of $h$ in $\Gamma$ is either $\Delta = \{h^{-1}\}$ or $\Delta = \{h^{-1}\}\cup N$, where  $N$ is a suborbit of size $|N|=q+1$.

	If $N=\varnothing$, then the statement in Lemma~\ref{lem:main}~\eqref{lem:main-item-3} holds. If $|N| = q+1$, then $\operatorname{C}_{G^*}(h)$ acts regularly on $N$, and so the subgraph of $\Gamma$ induced by $N$ is a Cayley graph of $\operatorname{C}_{G^*}(h) \cong \mathbb{Z}_{q+1}$.
	
	Let $\tilde{\Gamma}$ be the subgraph of $\Gamma$ induced by $N$. Assume that $\operatorname{C}_{G^*}(h) = \langle V\rangle\cong \mathbb{Z}_{q+1}$. 
	If $U \in N$, then the vertices of  $\tilde{\Gamma}$ are $V^iUV^{-i}$ for $0\leq i\leq q$, and a vertex is adjacent to $U$ if and only if $V^iUV^{-i} U^{-1}$ has order 3. Thus $\tilde{\Gamma}$ is isomorphic to $\operatorname{Cay}(\langle V\rangle,T)$, where 
	$$T = \left\{ V^i:0\leq i\leq q, \operatorname{Tr}(V^iUV^{-i}U^{-1}) = \pm 1 \right\}.$$
	As the trace is invariant by cyclic permutation, we also have $$\operatorname{Tr}(V^iUV^{-i}U^{-1}) = \operatorname{Tr}(UV^{-i}U^{-1}V^i) = \operatorname{Tr}\left(U\left(V^{-i}UV^i\right)^{-1}\right)$$
	for any $0\leq i\leq q.$

\section{The $1$-arc regular subgroup $G = \psl{2}{q}$}
	
	Let $X,G,G^*$ be as in Hypothesis~\ref{hyp} and $q\equiv 2\pmod 3$. Let $\Gamma$ be the subgraph of $\overline{\Gamma_{G}}$ induced by $\mathcal{C}_3$. Under the assumption in~\eqref{eq:assumption}, a maximum clique of $\Gamma$ contains $h$, so under this assumption, a canonical coclique of $\Gamma_{G}$ is $\langle h \rangle$; in $\Gamma$ a \itbf{canonical clique} is a clique equal to $\{h,h^{-1}\}$. We recall the following lemma.
	
	\begin{lem}\cite[Proposition~2.13]{hujdurovic2022intersection-cyclic}
		Let $q = p^k$ for a prime $p\neq 3$. Then, under the assumption \eqref{eq:assumption}, the non-canonical cocliques of $\Gamma_G$ do not contain the vertex-stabilizer $G_v = \langle h \rangle$.\label{lem:basic-int-set}
		\label{lem:nostab}
	\end{lem}

By Lemma~\ref{lem:main}, $\Gamma$ is equal to either the orbital graph of corresponding to the suborbit $\{h^{-1}\}$ or the union of two orbital graphs corresponding to the suborbits $\{h^{-1}\}$ and $N$, where $|N| = q+1$. If $\Gamma$ is the orbital graph corresponding to $\{h^{-1}\}$, then it is equal to a perfect matching. Therefore, $\omega(\Gamma) = 2$, and by~\eqref{eq:density}, we have $\rho(G) = 1$.

	If the neighbourhood of $h$ in $\Gamma$ is $\Delta = \{h^{-1}\}\cup N$, where $N$ is a suborbit of $\operatorname{C}_{G^*}(h)$, then by Lemma~\ref{lem:main}~\eqref{lem:main-item-3}, we know that $\tilde{\Gamma} = \Gamma[N]$ is a Cayley graph of $\operatorname{C}_{G^*}(h)$. 
Under the assumption \eqref{eq:assumption}, Lemma~\ref{lem:nostab} implies that  implies a non-canonical clique of $\Gamma$ cannot contain $h^{-1}$. This means that we can focus on $\omega(\tilde{\Gamma})$, since in this case
	\begin{align*}
		\rho(G) = \frac{2+\omega(\tilde{\Gamma})}{|G|/|V|}. 
	\end{align*}

	Let $U \in N$, from Equation~\ref{eq:conjofh}, and the comments following it,
	\begin{align*}
		U 		=
		\overline{ \begin{bmatrix}
			a & b-a-1\\
			b & -1-a
		\end{bmatrix} }
		\mbox{ and }
		U^{-1}
		=
		\overline{ \begin{bmatrix}
			-1-a & -b+a+1\\
			-b & a
		\end{bmatrix}}	
	\end{align*}
	for some $a\in \mathbb{F}_q$ and $b\in \mathbb{F}_q^*$ that satisfy $b^2-ab+a^2+a+1 = b$. As $N$ is an orbit of $\operatorname{C}_{G^*}(h)$ acting by conjugation, we can obtain all elements of $N$ by conjugating $U$. Recall that
	\begin{align}
	\operatorname{C}_{G^*}(h)
		= 
		\left\{
		\overline{  \begin{bmatrix}
			1 & \alpha\\
			-\alpha & 1+\alpha
		\end{bmatrix} }
		: 
		\alpha \in \mathbb{F}_{q}
		\right\}\cup \left\{ 
		\overline{  \begin{bmatrix}
			0 & -1\\
			1 & -1
		\end{bmatrix} } \right\}.
	\end{align}
	Assume that $\operatorname{C}_{G^*}(h) = \langle V\rangle$, where 
	$V = \overline{  \begin{bmatrix}
		1 & \alpha'\\
		-\alpha' & 1+\alpha'
	\end{bmatrix} }$, 
	for some $\alpha' \in \mathbb{F}_q$ and $V^\frac{q+1}{3} = h$. From this, we can see that $N = \{ V^iUV^{-i} : 0\leq i\leq q \}$. The next results consider the adjacencies between the elements $V^iUV^{-i}$.
	
	\begin{lem}
		For any $ i\in \{1,2\}$, the vertices $U$ and $h^iUh^{-i}$ are not adjacent in $\tilde{\Gamma}$.\label{lem:order3}
	\end{lem}
	\begin{proof}
		We note that
		\begin{align*}
			hUh^{-1}U^{-1} &= 
			\overline{ \begin{bmatrix}
				0 & 1\\
				-1 & 1
			\end{bmatrix}}
			\overline{ \begin{bmatrix}
				a & b-a-1\\
				b & -1-a
			\end{bmatrix}}
			\overline{ \begin{bmatrix}
				1 & -1\\
				1 & 0
			\end{bmatrix}}
			\overline{ \begin{bmatrix}
				-1-a & -b+a+1\\
				-b & a
			\end{bmatrix}} \\
			&=
			\overline{ \begin{bmatrix}
				a^{2} - {\left(a + 1\right)} b + b^{2} + 2 \, a + 1 & -a^{2} + {\left(a + 2\right)} b - b^{2} - 2 \, a - 1 \\
				a^{2} - a b + b^{2} + a & -a
			\end{bmatrix} } \\
			&= 
			\overline{  \begin{bmatrix}
				a & -a + b\\
				b-1 & -a
			\end{bmatrix}}.
		\end{align*}
		Hence, the trace of $hUh^{-1}U^{-1}$ is equal to $0$, so $U$ and $hUh^{-1}$ cannot be adjacent. Since the trace is invariant by cyclic permutation, we have $0 =\operatorname{Tr}(hUh^{-1}U^{-1}) = \operatorname{Tr}(Uh^{-1}U^{-1}h)= \operatorname{Tr}\left(U(h^{-1}Uh)^{-1}\right)$. Therefore, $U$ is also not adjacent to $h^{-1}Uh = h^2 U h^{-2}$. 	
\end{proof}
		
	Now, we consider the adjacency between $U$ and $V^iUV^{-i}$, where $i\neq \pm \frac{q+1}{3}$. Let $i \in \{0,\ldots,q\}\setminus \{ \tfrac{q+1}{3},\tfrac{2(q+1)}{3} \}$. Then, there exists $\alpha\in \mathbb{F}_q$ such that
	\begin{align*}
		V^i = 
		\overline{  \begin{bmatrix}
			1 & \alpha\\
			-\alpha & 1+\alpha
		\end{bmatrix} } .
	\end{align*}
	Note that the inverse of this element is
	\begin{align*}
		 \overline{(\alpha^2+\alpha+1)^{-1}
		\begin{bmatrix}
			1+\alpha & -\alpha\\
			\alpha &1
		\end{bmatrix}}.
	\end{align*}
	\begin{lem}
		The trace of $V^{i}UV^{-i}U^{-1}$ is $2(\alpha+1)(\alpha^2+\alpha+1)^{-1}$.
	\end{lem}
	\begin{proof}
		We have
		\begin{align*}
			V^{i}UV^{-i}U^{-1} &= 
			 \overline{  (\alpha^2+\alpha+1)^{-1}
			  \begin{bmatrix}
				1 & \alpha\\
				-\alpha & 1+\alpha
			\end{bmatrix} 
			\begin{bmatrix}
				a & b-a-1\\
				b & -1-a
			\end{bmatrix}
			\begin{bmatrix}
				1+\alpha & -\alpha\\
				\alpha &1
			\end{bmatrix}
			\begin{bmatrix}
				-1-a & -b+a+1\\
				-b & a
			\end{bmatrix} } \\
			&=
			 \overline{  (\alpha^2+\alpha+1)^{-1}
			\begin{bmatrix}
				a \alpha^{2} + {\left(a - b + 1\right)} \alpha + 1 & *\\
				* & -a\alpha^2 - (a - b - 1)\alpha + 1
			\end{bmatrix} }.
		\end{align*}
		Hence, 
		\begin{align*}
			\operatorname{Tr}(V^iUV^{-i}U^{-1}) &= (\alpha^2+\alpha+1)^{-1}(a \alpha^{2} + {\left(a - b + 1\right)} \alpha + 1-a\alpha^2 - (a - b - 1)\alpha + 1)\\
			&=2(\alpha+1)(\alpha^2+\alpha+1)^{-1}.
		\end{align*}		
	\end{proof}
	The following corollary is immediate.
	\begin{cor}
		The vertices $U$ and $V^iUV^{-i}$ are adjacent if $2(\alpha+1) = \pm (\alpha^2+\alpha+1)$.\label{cor:main}
	\end{cor}
	The next theorem is the main result of this section.
	\begin{thm}
		If $q\equiv 2\pmod 3$, then 
		\begin{align}\label{thm:main1}
			\rho(G) = 
			\begin{cases}
				\frac{4}{3} & \mbox{ if } q = 5^{2k+1}, \mbox{ for some integer }k\geq 0, \\
				1&\mbox{ if } q\equiv \pm 2 \pmod 5,\\
				\frac{4}{3} &\mbox{ if } q\equiv \pm 1 \pmod 5.
			\end{cases}
		\end{align}
	\end{thm}
	\begin{proof}
		From Lemma~\ref{lem:main}, the graph $\tilde{\Gamma}$ is a Cayley graph of $\operatorname{C}_{G^*}(h) = \langle V\rangle$. Using Corollary~\ref{cor:main}, we will show that the degree of $\tilde{\Gamma}$ is either 0, 1 or 2, by finding the number of solutions to $2(\alpha+1) = \pm (\alpha^2+\alpha+1)$.		
		
		First consider the equation $2(\alpha+1) = - (\alpha^2+\alpha+1)$, in this case $\alpha^2+3\alpha+3 = 0$ and
		\begin{align*}
			\left(2\alpha+3\right)^2 = -3,
		\end{align*}
		This has no solutions, since $-3$ is never a square in $\mathbb{F}_q$ when $q\equiv 2\pmod 3$.
				
		 Next consider the other equation, $2(\alpha+1) = (\alpha^2+\alpha+1)$. In this case $\alpha^2-\alpha-1 = 0$ and we have
		 \begin{align}
		 	\left(2\alpha-1\right)^2 = 5.\label{eq:final}
		 \end{align}
		 
		Since the above equation is quadratic in $\alpha$ and has coefficients in $\mathbb{F}_p$, its solutions are in $\mathbb{F}_{p^2}$. As $q$ is an odd power of $p$ (this follows from $q \equiv 2 \pmod 3$),  $\mathbb{F}_{p^2}$ is not a subfield of $\mathbb{F}_q$, and so it is clear that \eqref{eq:final} has a solution in $\mathbb{F}_q$ if and only if it has a solution in $\mathbb{F}_p$. 
		
		 First consider the case where $p=5$, then $q = 5^{2k+1}$ for some $k\geq 0$. Equation \eqref{eq:final} becomes $2\alpha-1 = 0$, so there is only one solution, namely $\alpha = 2^{-1} = 2$. In this case, the degree of $\tilde{\Gamma}$ is 1, so it is a union of edges. Therefore,
		 \begin{align*}
		 	\rho(G) = \frac{2+\omega(\tilde{\Gamma})}{|G|/|V|} = \frac{4}{3}.
		 \end{align*}
		
		Next, assume that $p\neq 5$. Using the Legendre symbol, since $p\neq 5$, we have
		\begin{align*}
			\left(\frac{5}{p}\right)=
			\begin{cases}
				1 &\mbox{ if }p\equiv \pm 1 \pmod 5,\\
				-1 &\mbox{ if }p\equiv \pm 2 \pmod 5.
			\end{cases}
		\end{align*}
		Hence, $5$ is a square in $\mathbb{F}_p$ if and only if $p\equiv \pm 1\pmod 5.$ 
		
		If $p\equiv \pm 2\pmod 5$, then $\left(2\alpha-1\right)^2 = 5$ has no solutions and $\tilde{\Gamma}$ is a coclique. Therefore, in this case
		\begin{align*}
			\rho(G) = \frac{2+\omega(\tilde{\Gamma})}{|G|/|V|} = \frac{3}{3}=1.
		\end{align*}
		
		Finally, if $p\equiv \pm 1\pmod 5$, then $\left(2\alpha-1\right)^2 = 5$ has exactly 2 solutions and $U$ has valency $2$ in $\tilde{\Gamma}$, implying $\tilde{\Gamma}$ is a union of cycles. We will show that these cycles can never have length $3$. Indeed, if a component of $\tilde{\Gamma}$ is a cycle of length $3$, then all components of $\tilde{\Gamma}$ are cycles of length $3$ due to fact that $\tilde{\Gamma}$ is a Cayley graph. In particular, the vertices in the component containing $U$ would be exactly $U$, $V^{\tfrac{q+1}{3}}UV^{-\tfrac{q+1}{3}} = hUh^{-1}$, and $V^{\tfrac{2(q+1)}{3}}UV^{-\tfrac{2(q+1)}{3}} = h^{-1}Uh$. But by Lemma~\ref{lem:order3}, we obtain a contradiction. Hence, the components of $\tilde{\Gamma}$ are cycles of length at least $4$. Therefore,
		\begin{align*}
			\rho(G) = \frac{2+\omega(\tilde{\Gamma})}{|G|/|V|} = \frac{4}{3}.
		\end{align*}
		This completes the proof.		
	\end{proof}

\section{The $2$-arc regular group $G^* = \pgl{2}{q}$}
	
Again, we let $X,G,G^*$ be as in Hypothesis~\ref{hyp}. 
In this section, we determine all the possible intersection densities of $\rho(G^*)$. We first recall that $\rho(G^*)\leq \rho(G)$. By Theorem~\ref{thm:main1}, we conclude that $\rho(G^*) = 1$, whenever $p\equiv \pm 2\pmod 5$. Therefore, we will assume that $p\equiv a\pmod 5$, where $a\in \{0,\pm 1\}$.
	
	We recall from Theorem~\ref{thm:main1} that $\alpha(\Gamma_G)\leq 4$. As $G\leq G^*$ and $\nv \in G^*\setminus G$ is an involution, we know that $G^* = G\rtimes \langle \nv \rangle$. Let $\mathcal{F} \subset G^*$ be an intersecting set. Then, it is clear that $|\mathcal{F}\cap G|\leq 4$ and $|\mathcal{F}\cap G \nv |\leq 4$. Consequently, we deduce $|\mathcal{F}| \leq 8$ and thus $\rho(G^*) \in \{1,\tfrac{7}{6},\tfrac{4}{3}\}$. We will show now that $\rho(G^*) = 1$.
	
	Recall from Lemma~\ref{lem:trace0}, if $\overline{A} \in G^*$ is an involution fixing a vertex, then $\operatorname{Tr}(A) = 0$.
	
		\begin{thm}
		If $p \equiv 0,\pm 1 \pmod 5$, then $\rho(G^*) = 1$.
	\end{thm}
	\begin{proof}
		Without loss of generality, let $\mathcal{F}$ be an intersecting set of $G^*$. We note that the vertex-stabilizer $G^*v = \langle h, \nv \rangle$ has order $6$, so $\alpha(\Gamma_{G}^*)\geq 6$. Let $\mathcal{F} \subset G^*$ be an intersecting set. We will show by contradiction that if $|\mathcal{F} \cap G\nv | = 4$, then $|\mathcal{F}\cap G| \leq 1$.
		
		Assume that $|\mathcal{F} \cap G \nv | = 4$ and $|\mathcal{F} \cap G|\geq 2$. As $|\mathcal{F}\cap G|\geq 2$, we may further assume that $\mathcal{F}$ contains $1$ (multiply $\mathcal{F}$ by the inverse of an element from $\mathcal{F} \cap G$, so that, we still assume without loss of generality that $|\mathcal{F} \cap G\nv| = 4$ and $|\mathcal{F} \cap G|\geq 2$.) Further, by taking a conjugate, we also assume $h \in \mathcal{F}$. Let $\mathcal{F}\cap G\nv = \{ X,Y,Z,W \}$. Since $X,Y,Z,W$ fix a vertex (equivalently, are adjacent to $1$), by Lemma~\ref{lem:trace-involution}, we must have 
		\begin{align*}
			X = 
			\overline{\begin{bmatrix}
				x & t_x\\
				r_x & -x
			\end{bmatrix}},\
			Y = 
			\overline{\begin{bmatrix}
					y & t_y\\
					r_y & -y
			\end{bmatrix}},\
			Z = 
			\overline{\begin{bmatrix}
					z & t_z\\
					r_z & -z
			\end{bmatrix}},\
			W = 
			\overline{\begin{bmatrix}
					w & t_w\\
					r_w & -w
			\end{bmatrix}},
		\end{align*}
		for some $x,y,z,w \in \mathbb{F}_q$ and $r_u, t_u\in \mathbb{F}_q$ for $u\in \{x, y, z, w\}$. Moreover, since $X$ is also adjacent to $h$ in $\overline{\Gamma_{G^*}}$, using the fact that $h\in G$ and $X\in G^*\setminus G$, we have $\operatorname{Tr}(Xh) = 0$. By noting that
		\begin{align*}
			Xh=
			\overline{
			\begin{bmatrix}
				t_x & *\\
				* & x-r_x
			\end{bmatrix}},
		\end{align*}
		we deduce that $r_x = t_x+x$. Hence, we have
		\begin{align*}
			X = 
			\overline{\begin{bmatrix}
					x & t_x\\
					t_x+x & -x
			\end{bmatrix}}.
		\end{align*}
		Similarly, we also have
		\begin{align*}
			Y = 
			\overline{\begin{bmatrix}
					y & t_y\\
					t_y+y & -y
			\end{bmatrix}},\
			Z = 
			\overline{\begin{bmatrix}
					z & t_z\\
					t_z+z & -z
			\end{bmatrix}},\
			W = 
			\overline{\begin{bmatrix}
					w & t_w\\
					t_w+w & -w
			\end{bmatrix}}.
		\end{align*}
		\begin{claim}
			One of $x,y,z,w$ is equal to $0$.\label{claim:pgl1}
		\end{claim}
		\begin{proof}[Proof of Claim~\ref{claim:pgl1}]
		If $x,y,z,w$ are all non-zero, then we have
		\begin{align*}
			X = 
			\overline{\begin{bmatrix}
					1 & a_x\\
					a_x+1 & -1
			\end{bmatrix}},\
			Y = 
			\overline{\begin{bmatrix}
					1 & a_y\\
					a_y+1 & -1
			\end{bmatrix}},\
			Z = 
			\overline{\begin{bmatrix}
					1 & a_z\\
					a_z+1 & -1
			\end{bmatrix}},\
			W = 
			\overline{\begin{bmatrix}
					1 & a_w\\
					a_w+1 & -1
			\end{bmatrix}}.
		\end{align*}
		As $X$ and $Y$ are adjacent in $\overline{\Gamma_{G^*}}$, we have $\operatorname{Tr}(YX) = \pm 1$. By noting that
		\begin{align*}
			YX =
			\overline{\begin{bmatrix}
				1+a_xa_y+a_x&*\\
				* & 1+a_xa_y+a_y
			\end{bmatrix}},
		\end{align*}
		we conclude, $2+2a_xa_y +a_x+a_y = \pm 1$. Similarly, we can show that
		\begin{align*}
			2+2a_xa_z +a_x+a_z &= \pm 1\\
			2+2a_xa_w +a_x+a_w &= \pm 1.
		\end{align*}
		Since these three equations are equal to $1$ or $-1$, without loss of generality we may assume that $2+ 2 a_x a_y +a_x+a_y = 2+2 a_x a_z +a_x+a_z.$ This implies that
		\begin{align}
			(2a_x + 1)a_y = (2a_x+1)a_z.\label{eq:equality}
		\end{align}
		If $a_x \neq -2^{-1}$, then $a_y = a_z$ which means $Y=Z$, so we must have $a_x =-2^{-1}$. In this case we have
		\begin{align*}
			2+2a_xa_z +a_x+a_z = 2-2^{-1}.
		\end{align*}
		We know that $2-2^{-1} \neq 1$, so $2-2^{-1} = -1$. The latter happens if and only if $p = 5$. Therefore, if $p \neq 5$, then \eqref{eq:equality} does not hold which is a contradiction with $X$ adjacent to $Y$. So one of $x, y, z, w$ must be equal to $0$.
		
		Consider the case when $p = 5$ and $a_x = -2^{-1} = 2$, then none of $a_y, a_z, a_w$ can also be equal to $2$. The adjacency between $Y,Z,$ and $W$ yield the equations
		\begin{align*}
			2+a_y+a_z+2a_ya_z&= \pm 1\\
			2+a_y+a_w+2a_ya_w&= \pm 1\\
			2+a_z+a_w+2a_za_w &= \pm 1.
		\end{align*}
		Again, two of these equations must be both equal to $1$ or $-1$, so we may assume without loss of generality that $2+a_y+a_z+2a_ya_z = 2+a_y+a_w+2a_ya_w$. Therefore, we have
		\begin{align*}
			(2a_y+1)a_z = (2a_y+1)a_w.
		\end{align*}
		As $a_y\neq 2$, we must have $a_z = a_w$ or equivalently $Z = W$. This is a contradiction and we conclude that exactly one of $x, y, z, w$ is equal to $0$.
		\end{proof}
		
		From Claim~\ref{claim:pgl1}, we may assume without loss of generality that $x = 0$, and so
		\begin{align*}
			X = \overline{\begin{bmatrix} 				0 & 1\\ 				1 & 0 		\end{bmatrix}}.
		\end{align*}
		\begin{claim}
			$X$ cannot be adjacent to $Y,Z,$ and $W$.\label{claim:pgl2}
		\end{claim}
		\begin{proof}[Proof of Claim~\ref{claim:pgl2}]
			Clearly, we have 
			\begin{align*}
				\operatorname{Tr}(XY) &= 2a_y+1 =\pm 1\\
				\operatorname{Tr}(XZ) &= 2a_z+1 = \pm 1\\
				\operatorname{Tr}(XW) &= 2a_w+1 = \pm 1.
			\end{align*}
			We deduce that $a_y,a_z,a_w$ are not distinct, implying that $Y,Z, W$ are not distinct, which is a contradiction. 
		\end{proof}
		
		We deduce from Claim~\ref{claim:pgl2} that $|\mathcal{F}\cap G \nv |\leq 3$, contradicting the assumption that $|\mathcal{F}\cap G \nv | = 4$. Consequently, if $|\mathcal{F}\cap G \nv | = 4$, then $|\mathcal{F}\cap G| \leq 1$. This shows that any intersecting set $\mathcal{F}$ of $G^*$ of maximum size must have the property that $|\mathcal{F}| = 6$, and $|\mathcal{F}\cap G| = |\mathcal{F}\cap G \nv | = 3$. This completes the proof.

	\end{proof}

%

\end{document}